\documentclass[12pt,leqno]{amsart}%
\allowdisplaybreaks
\usepackage{amsmath}
\usepackage{mathtools}
\mathtoolsset{showonlyrefs} 
\usepackage{esint}

\usepackage[margin=1.5in]{geometry}
\usepackage{graphicx}
\usepackage{amsfonts}
\usepackage{amssymb}
\usepackage[normalem]{ulem}
\usepackage[utf8]{inputenc}
\usepackage{comment}
\usepackage{hyperref}%
\usepackage{amsthm}
\usepackage{url}
\usepackage[dvipsnames]{xcolor}
\usepackage[english=american]{csquotes}
\usepackage{enumerate}
\usepackage{textcomp}
\usepackage{mathrsfs}
\usepackage{color, colortbl}
\definecolor{Gray}{gray}{0.9}
\usepackage{bbm}
\usepackage{stmaryrd}
\providecommand{\U}[1]{\protect\rule{.1in}{.1in}}
\newtheorem{theorem}{Theorem}[section]
\theoremstyle{plain}

\newtheorem{corollary}{Corollary}[section]

\newtheorem{lemma}{Lemma}[section]

\newtheorem{proposition}{Proposition}[section]

\numberwithin{equation}{section}
\theoremstyle{definition}

\theoremstyle{remark}

\begin{document}
\title[Interior Estimates]{
A Liouville type theorem for ancient Lagrangian Mean Curvature flows}
\author{Arunima Bhattacharya, Micah Warren, and Daniel Weser}

\address{Department of Mathematics, Phillips Hall\\
 the University of North Carolina at Chapel Hill, NC }
\email{arunimab@unc.edu}

\address{Department of Mathematics\\
	University of Oregon, Eugene, OR 97403}

\email{micahw@uoregon.edu}

\address{Department of Mathematics, Phillips Hall\\
 the University of North Carolina at Chapel Hill, NC }
 \email{weser@unc.edu}

\begin{abstract} 
We prove a Liouville type result for convex solutions of the Lagrangian mean curvature flow with restricted quadratic growth assumptions at antiquity on the solutions.
\end{abstract}
\maketitle

\section{Introduction}

A family of Lagrangian submanifolds $X(x,t):\mathbb R^n\times\mathbb R \to\mathbb  C^n$ is said to evolve by mean curvature if
\[
(X_t)^\bot=\Delta_gX=\vec H,
\]
where $\vec H $ denotes the mean curvature vector of the Lagrangian submanifold. After a change of coordinates, one can locally write $X(x,t)=(x,Du(x,t))$ such that $\Delta_gX=J\nabla_g\Theta$ (see \cite{HL}): Here $g=I_n+(D^2u)^2$ is the induced metric on $(x,Du(x))$, $J$ is the almost complex structure on $\mathbb C^n$,
and $\Theta$ is the Lagrangian angle given by
\begin{equation*}
    \Theta=\sum_{i=1}^n\arctan\lambda_i ,
\end{equation*}
where $\lambda_i$ are the eigenvalues of the Hessian $D^2u$. This
results in a local potential $u(x,t)$ evolving by the parabolic equation
\begin{align}
    u_t&=\sum_{i=1}^n\arctan\lambda_i\label{ut}.
\end{align}
We are free to add any function that does not depend on $x$ to the right-hand side, since this does not change the flow of the gradient graph of $u$. Thus, we will consider the equation 
\begin{equation}
    u_t =\sum_{i=1}^n\arctan\lambda_i - \Theta_0 \label{ut2}
\end{equation}
for a fixed $\Theta_0$. In particular, under this flow, solutions to the special Lagrangian equation
\begin{equation} 
    \sum_{i=1}^n\arctan\lambda_i  = \Theta_0
\end{equation} 
will be stationary.

Our main result in this paper is the following:
\begin{theorem}
Let $u$ be an entire ancient solution to \eqref{ut2}.
Suppose that $u$ is convex and satisfies the following growth condition at antiquity: there exists an $R_{0}$ such that
\[
\lim\sup_{t\rightarrow-\infty}\sup_{x\in\mathbb{R}^{n}}\frac{\left\vert
u(x,t)\right\vert }{\left\vert x\right\vert ^{2}+R_{0}}<\frac{1}{(6\sqrt{n}+2)^2}.
\]
Then $u$ is a quadratic polynomial \  
\end{theorem}

In essence, this theorem states that if a convex ancient solution has the growth of a small quadratic polynomial for times far enough in the past, then the solution must in fact be a quadratic polynomial. Although the growth condition is somewhat restrictive, it applies to the scalar potential $u$, not to the gradient graph, and requires only convexity without additional constraints on the Hessian.

We derive a Hessian estimate for solutions of \eqref{ut} via a pointwise approach assuming a quadratic growth condition at antiquity. 
The method used here follows a parabolic adaptation of Korevaar's maximum principle method as in Evan-Spruck \cite{evans1992motion}.  A similar maximum principle estimate for the elliptic special Lagrangian equation was carried out in Warren-Yuan \cite{warren2008liouville}, inspired by Korevaar's \cite{K} estimate for minimal surface equations.   
As in \cite{warren2008liouville}, our method requires a smallness condition on the growth.
Once a Hessian bound is established,  H\"older regularity and the Liouville property follow from Nguyen-Yuan \cite{YYmcf}.

In 1968, Bombieri, De Giorgi, and Miranda \cite{bombieri1969}
 established an interior estimate bound for classical
solutions of the minimal surface equation in $\mathbb{R}^n$.
In the 1970s, Trudinger \cite{trudinger1972new} provided a new and simpler derivation of this
estimate and partly developed in the process some new
techniques applicable to the study of hypersurfaces in
general. In the 1980s, Korevaar \cite{K} found a
strikingly simple pointwise argument. Korevaar's technique was adapted to prove estimates for the parabolic version, namely mean curvature flow, by Ecker-Huisken \cite{EK}, Evans-Spruck \cite{evans1992motion}, Colding-Minicozzi \cite{colding2004sharp} in codimension one.

Mean curvature flow in higher codimensions including the Lagrangian mean curvature flow has been studied by many authors \cite{CCY,CCY13, SmW, wang2002long, tsai2023mean,wan2024sharp, lotay2024ancient}. However, in higher codimensions, an adaption of Korevaar's method with no restriction on the height remains elusive, as the diagonization step is difficult to overcome.  To diagonalize the differential of a vector-valued function by rotating the base, one requires that the gradients of each of the functions in the vector be mutually orthogonal, which fails in general.  Interior Hessian estimates for the elliptic special Lagrangian equation and the variable phase Lagrangian mean curvature equation have been established via an integral approach in \cite{BC,BCGJ, WY2d, WY9, CWY, WY,WaY, CSY, AB, AB2d, BMS, ZhouHess, BWall}, via a compactness approach in \cite{Lcomp}, via a doubling approach in \cite{shankar2024hessian}. The Liouville result of \cite{warren2008liouville} is generalized to remove the small quadratic constraint via a compactness method in \cite{QiDing}. A pointwise approach to proving interior Hessian estimates without any constraints on the growth remains an open question even for the elliptic special Lagrangian equation.

\medskip
\noindent\textbf{Acknowledgments.} 
AB acknowledges
the support of the Simons Foundation grant MP-TSM-00002933. DW acknowledges
the support of the NSF RTG DMS-2135998 grant.

\section{Jacobi Inequality}

Let $g$ denote the induced metric of the embedding $X$, which in the graphical case takes the form $g=I_n+D^2u(\cdot,t)D^2u(\cdot,t)$. 
If we take a derivative of (\ref{ut2}), we find 
\begin{equation*}
    u_{tk}=g^{ij}u_{ijk},
\end{equation*} 
so that, if we define the operator
\begin{equation*}
    L := \partial_t - g^{ij} \partial_{ij},
\end{equation*} 
we see that equivalently
\begin{equation}\label{Lu_k}
    Lu_{k}=0.
\end{equation}

We will let $V=\sqrt{|\det g|}$ denote the volume element, and we will define $b=V^{1/n}$ to be the volume element raised to the power of $1/n$. In the following lemma, we show that a Jacobi inequality holds along the ``heat flow'' of $b$ with respect to $L$, which will be a crucial ingredient in the proof of the main theorem.

\begin{lemma}
    For a convex solution of \eqref{ut}, the volume element $b=(\sqrt{|\det g|})^{1/n}$ satisfies 
    \begin{equation}\label{jacobi ineq}
        Lb + 2 \frac{|\nabla_g b|^2}{b} \leq0.
    \end{equation}
\end{lemma}

\medskip

\begin{proof}
    We directly compute
    \begin{align}
        \partial_t b &= \frac{1}{2n}  \mathrm{tr}(g^{-1}\partial_t g)  b \\
        \partial_i b &= \frac{1}{2n}  \mathrm{tr}(g^{-1}\partial_i g)  b \\
        \partial_{ij} b &= \frac{1}{2n} \mathrm{tr}(g^{-1}\partial_{ij} g)  b + \frac{1}{2n} \mathrm{tr}(\partial_jg^{-1}\partial_i g)  b + \frac{\partial_i b  \partial_j b}{b} .
    \end{align}
    Therefore, we find the initial expansion
    \begin{equation}
        Lb = \frac{1}{2n} \mathrm{tr}(g^{-1} L g)  b - \frac{1}{2n} g^{ij}\mathrm{tr}(\partial_jg^{-1}\partial_i g)  b - \frac{|\nabla_gb|^2}{b} .\label{Lb 1}
    \end{equation}
    Now we will compute $Lg$. To do so, we fix a point $p$ and rotate coordinates so that $D^2u$ is diagonalized at $p$. (Note; a rotation in the domain of a scalar function generates a transpose rotation of the gradient of that function. 
 The conjugation by these is what diagonalizes the Hessian.)  We denote by $\lambda_k$ the eigenvalues of $D^2u(p)$, so that $D^2u(p)=\textrm{diag}(\lambda_1,\hdots,\lambda_n)$. Note that  $\lambda_k\geq0$ by our convexity hypothesis. In these coordinates at the point $p$, the metric $g$ takes the form 
    \begin{equation}
        g_{kl} = \delta_{kl} + Du_k\cdot Du_l = (1+\lambda_k^2)\delta_{kl},
    \end{equation}
    and thus 
    \begin{equation}
        g^{kl} = (1+\lambda_k^2)^{-1}\delta^{kl}.
    \end{equation}
    
    Applying these identities at the point $p$, one obtains
    \begin{align}
        \partial_t g_{kl} 
            &= (\lambda_k+\lambda_l)  \partial_t u_{kl} \\
        \partial_{ij} g_{kl} &= (\lambda_k+\lambda_l)  u_{ijkl} + Du_{ik}\cdot Du_{jl} + Du_{jk}\cdot Du_{il} ,
    \end{align}
    so that
    \begin{align}
        Lg_{kl} 
            &= (\lambda_k+\lambda_l)  \partial_t u_{kl} - g^{ij}\Big((\lambda_k+\lambda_l)  u_{ijkl} + Du_{ik}\cdot Du_{jl} + Du_{jk}\cdot Du_{il}\Big) \\ 
            &= (\lambda_k+\lambda_l)  \partial_t u_{kl} - g^{ii}(\lambda_k+\lambda_l)  u_{iikl} + 2g^{ii}Du_{ik}\cdot Du_{il} .
    \end{align}
    Thus, at the point $p$ we find  
    \begin{align}
        \mathrm{tr}(g^{-1}Lg) 
            &= g^{kk}Lg_{kk} \\
            &= 2g^{kk}\lambda_k  \partial_t u_{kk} - 2g^{ii}g^{kk}\lambda_k u_{iikk} - 2g^{ii}g^{kk}|Du_{ik}|^2 \\
            &= 2g^{kk}\lambda_kL u_{kk} - 2g^{ii}g^{kk}|Du_{ik}|^2 \label{tr g Lg 1} .
    \end{align}
    Now we compute $Lu_{kk}$. Differentiating $Lu_k=0$, we obtain
    \begin{equation}
        Lu_{kk} = (\partial_k g^{ij})u_{ijk}.
    \end{equation}
    Differentiating the identity $\delta_a^c=g_{ab}g^{bc}$, one obtains at the point $p$ that
    \begin{equation}\label{pa_k g^ij}
        \partial_k g^{ij} = -g^{ii}g^{jj}(\lambda_i+\lambda_j)u_{ijk} ,
    \end{equation}
    so that 
    \begin{align}
        Lu_{kk} 
            &= (\partial_k g^{ij})u_{ijk} \\
            &= -g^{ii}g^{jj}(\lambda_i+\lambda_j)u_{ijk}^2 . \label{Lu_kk}
    \end{align}
    Combining \eqref{tr g Lg 1} and \eqref{Lu_kk}, we find at the point $p$ that
    \begin{align}
        \mathrm{tr}(g^{-1}Lg) 
            &= 2g^{kk}\lambda_kL u_{kk} - 2g^{ii}g^{kk}|Du_{ik}|^2 \\
            &= -2g^{ii}g^{jj}g^{kk}(\lambda_i+\lambda_j)\lambda_kL u_{ijk}^2 - 2g^{ii}g^{kk}|Du_{ik}|^2 \label{tr g^ij Lg_ij} .
    \end{align}

    Next, using \eqref{pa_k g^ij}, we directly compute
    \begin{align}
        g^{ij}\mathrm{tr}(\partial_jg^{-1}\partial_i g) 
            &= g^{ii} \partial_i g^{kl} \partial_i g_{kl} \\
            &= -g^{ii}g^{kk}g^{ll}(\lambda_k+\lambda_l)^2u_{ikl}^2, \label{g^ij tr pa_j}
    \end{align}
    so that, inserting \eqref{tr g^ij Lg_ij} and \eqref{g^ij tr pa_j} into \eqref{Lb 1}, we obtain
    \begin{align}
        Lb 
            &= \frac{1}{2n} \mathrm{tr}(g^{-1} L g)  b - \frac{1}{2n} g^{ij}\mathrm{tr}(\partial_jg^{-1}\partial_i g)  b - \frac{|\nabla_gb|^2}{b} \\
            &= \frac{b}{2n} \bigg(-2g^{ii}g^{jj}g^{kk}(\lambda_i+\lambda_j)\lambda_k u_{ijk}^2 - 2g^{ii}g^{kk}|Du_{ik}|^2 \\
            &\hspace{2cm} + g^{ii}g^{kk}g^{ll}(\lambda_k+\lambda_l)^2u_{ikl}^2\bigg) - \frac{|\nabla_gb|^2}{b} \\
            &= \frac{b}{2n} \bigg(-2g^{ii}g^{jj}g^{kk}(\lambda_i+\lambda_j)\lambda_k u_{ijk}^2 - 2g^{ii}\delta^{jj}g^{kk}u_{ijk}^2 \\
            &\hspace{2cm} + g^{ii}g^{kk}g^{ll}(\lambda_k+\lambda_l)^2u_{ikl}^2\bigg) - \frac{|\nabla_gb|^2}{b} \\
            &= \frac{b}{2n} \bigg(-2g^{ii}g^{jj}g^{kk}(\lambda_i+\lambda_j)\lambda_k u_{ijk}^2 - 2g^{ii}g^{jj}g^{kk}(1+\lambda_j^2)u_{ijk}^2  \\
            &\hspace{2cm} + g^{ii}g^{kk}g^{ll}(\lambda_k+\lambda_l)^2u_{ikl}^2\bigg) - \frac{|\nabla_gb|^2}{b} \\
            &= \frac{b}{2n} \bigg(-4g^{ii}g^{jj}g^{kk}\lambda_i\lambda_k u_{ijk}^2 - 2g^{ii}g^{jj}g^{kk}\lambda_j^2u_{ijk}^2 - 2g^{ii}g^{jj}g^{kk}u_{ijk}^2 \\
            &\hspace{2cm} + 2g^{ii}g^{kk}g^{ll}\lambda_k^2u_{ikl}^2 + 2g^{ii}g^{kk}g^{ll}\lambda_k\lambda_lu_{ikl}^2\bigg) - \frac{|\nabla_gb|^2}{b} \\ 
            &= \frac{b}{2n} \bigg(-2g^{ii}g^{jj}g^{kk}\lambda_i\lambda_k u_{ijk}^2 - 2g^{ii}g^{jj}g^{kk}u_{ijk}^2 \bigg) - \frac{|\nabla_gb|^2}{b} . \label{Lb 2}
    \end{align}
    Finally, we use Cauchy-Schwarz to see that 
    \begin{align}
        \frac{|\nabla_gb|^2}{b} = \frac{b}{n^2}  g^{ii}\Big(g^{kk}\lambda_ku_{ikk}\Big)^2 \leq \frac{b}{n}  g^{ii}g^{kk}g^{kk}\lambda_k^2u_{ikk}^2,
    \end{align}
    which, by comparing with \eqref{Lb 2} and using the fact that $\lambda_k\geq0$ by convexity, concludes the poof of the lemma.
\end{proof}

\bigskip

\section{Oscillation and gradient bounds forward in time}

In this section, we will prove interior oscillation and gradient bounds forward in time for convex solutions of \eqref{ut2} defined on parabolic cylinders of the form
\[
\overline{B}_{R}(0)\times\lbrack0,\frac{1}{n}]\subset\mathbb{R}^{n}\times\mathbb{R}.%
\]
First, we prove an interior height bound forward in time for general solutions of \eqref{ut}.

\begin{proposition}
Suppose that $u$ is a solution of (\ref{ut}) on $\overline{B}_{R}(0)\times
\lbrack0,\frac{1}{n}]$. \ \ Then we have
\begin{equation}\label{int height bound}
u(0,\frac{1}{n})\leq\arctan\left(  \frac{\pi}{R^{2}}\right)  +\max
_{ \overline{B}_{R}(0) \times\left\{  0\right\}  }u\left(  x,0\right).
\end{equation}

\end{proposition}

\begin{proof} Up to subtracting a constant, we will assume that $\max u=0$ at $t=0.$ \ Notice that by simply integrating equation (\ref{ut}) at each point on $\partial B_{R}(0)$ we may obtain the bound
\begin{equation}
u(x,t)\leq\frac{tn\pi}{2}\text{ on }\partial B_{R}(0)\times\lbrack
0,\frac{1}{n}].\label{b1}%
\end{equation}

We will construct a supersolution to obtain the interior bound. Let
\[
w=\frac{tn\pi}{2}\left(  \frac{\left\vert x\right\vert }{R}\right)
^{2}+nt\arctan\left(  \frac{tn\pi}{R^{2}}\right),
\]
which satisfies
\[
w_{t}-\sum\arctan\lambda_{i}\geq0.
\]
By \eqref{b1}, $w \geq u$ on the boundary, and hence also in the
interior of $\overline{B}_{R}(0)\times\lbrack0,\frac{1}{n}]$. \ Thus
\begin{align*}
u(0,\frac{1}{n})  & \leq w(0,\frac{1}{n})\leq\arctan\left(  \frac{\pi}{R^{2}%
}\right) .
\end{align*}
\end{proof}

\begin{corollary}\label{coro}
Suppose for any $\Theta_0$ the function $u$ is a convex solution of (\ref{ut2}) on
\[
\overline{B}_{2R+1}(0)\times\lbrack0,\frac{1}{n}]\subset\mathbb{R}^{n}\times\mathbb{R}.%
\]
Then,
\[
\max_{\overline{B}_{1}(0)\times\lbrack0,\frac{1}{n}]}|Du|\leq\frac{1}{R}\left[
M+\arctan\left(  \frac{\pi}{  R^{2}}\right)  \right]
\]
where $M$ is the oscillation of $u$ at $t=0.$
\end{corollary}

\begin{proof}
By subtracting a constant, we will assume that $\max_{\overline{B}_{2R+1}%
(0)\times\lbrack0,T)}u=0$, so that   $\min_{\overline{B}_{2R+1}%
(0)\times\lbrack0,T)}u = -M.$ Adding $\Theta_0 t$, we may also assume that $u$ is a solution of (\ref{ut}). From \eqref{int height bound}, we have 
\[
u(x,t)\leq\arctan\left(  \frac{\pi}{  R^{2}}\right)  \text{
on }B_{R+1}(0)\times\lbrack0,\frac{1}{n}].
\]

Since $u$ is convex, we know $\Theta\geq0$, and hence $u$ can only increase along the flow of type (\ref{ut}). Therefore, we see that the oscillation is no more than $M+\arctan\left(  \frac{\pi}{  R^{2}}\right)  $ on $B_{R+1}(0)\times\lbrack0,\frac{1}{n}].$ \ The conclusion follows by applying the convexity condition to any point on $B_{1}(0)$, noting that the derivative bound from convexity depends only on the oscillation of each time slice. However, since solutions of (\ref{ut}) and (\ref{ut2}) only differ by a constant in each time slice, we also obtain the conclusion for solutions of (\ref{ut2}). 
\end{proof}

\section{Hessian bound}

In this section, we prove an interior Hessian bound forward in time for convex solutions of \eqref{ut} and \eqref{ut2}. The first proposition applies to convex solutions with sufficiently small gradient $|Du|$, while the second applies to convex solutions with bounded oscillation.

\begin{proposition}\label{Hb}
Let $u$ be a convex solution of \eqref{ut} or \eqref{ut2} on $B_{1}(0)\times\lbrack0,\frac{1}{n}]$. Suppose that for some $\alpha,\gamma >0$ we have
\[
\alpha> \frac{3n}{2},
\]
\[
\alpha\gamma<1, 
\]
and
\[
\left\vert Du\right\vert ^{2}<\gamma\text{ on }B_{1}(0)\times\lbrack0,\frac{1}{n}].
\]
Then for any $K>0$%
\[
\lambda_{\max}^{2}(0,\frac{1}{n})\leq\frac{e^{2nK}\left(  \frac{2\alpha}{2\alpha-
3n   }\right)  ^{n}}{\left(  e^{K\left(  1-\alpha
\gamma\right)  }-1\right)  ^{2n}}.
\]

\end{proposition}

\begin{proof}
Let
\begin{align*}
b &  =V^{1/n}\\
\phi &  =\left[  \alpha\left\vert Du\right\vert ^{2}-\alpha\gamma+nt
(1-\left\vert x\right\vert ^{2})\right]  ^{+}\\
f &  =e^{K\phi}-1\\
\eta &  =f\circ\phi,
\end{align*}
and consider the function
\[
h=\eta b.
\]
Notice that
\[
-\alpha\gamma\leq\alpha\left\vert Du\right\vert ^{2}-\alpha\gamma\leq0,
\]
so $\phi$ is 0 when $t=0$ (and thus so is $\eta$). Additionally, $\phi$ vanishes on the boundary $\partial B_1(0)$ for any time. \ 

Now, since $-\alpha\gamma+1>0$, the function $\eta$ will be positive somewhere, and
we let $\left(  x_{0},t_{0}\right)$ be a point where a positive maximum for
$h$ occurs on $B_{1}(0)\times\lbrack0,\frac{1}{n}]$. \ \ Then we must have
\[
Lh=h_{t}-g^{ij}h_{ij}\geq0.
\]
From spatial maximality, we obtain
\[
\eta_{i}=-\frac{\eta b_{i}}{b},
\]
which implies
\begin{align*}
Lh  & =\eta_{t}b+\eta b_{t}-g^{ij}\eta_{ij}b-2g^{ij}\eta_{i}b_{j}-\eta
g^{ij}b_{ij}\\
& =\left(  \eta_{t}-g^{ij}\eta_{ij}\right)  b+\eta\left(  b_{t}+2g^{ij}%
\frac{b_{i}}{b}b_{j}-g^{ij}b_{ij}\right)  .
\end{align*}
By \eqref{jacobi ineq}, we know that
\[
\left(  b_{t}+2g^{ij}\frac{b_{i}}{b}b_{j}-g^{ij}b_{ij}\right)  \leq0,
\]
so we conclude at $\left(  x_{0},t_{0}\right)  $ that 
\[
\eta_{t}-g^{ij}\eta_{ij}\geq0.
\]
We compute
\begin{align*}
\eta_{i} &  =Ke^{K\phi}\left(  2\alpha\sum_{k=1}^{n}u_{k}u_{ki}-2 n t x_{i}\right)  \\
\eta_{ij} &  =Ke^{K\phi}\left(  2\alpha\sum_{k=1}^{n}u_{kj}u_{ki}+2\alpha
\sum_{k=1}^{n}u_{kij}u_{k}-2 n t\delta_{ij}\right)  \\
& \hspace{.5cm} +K^{2}e^{K\phi}\left(  2\alpha\sum_{k=1}^{n}u_{k}u_{ki}-2 n t 
x_{i}\right)  \left(  2\alpha\sum_{k=1}^{n}u_{k}u_{kj}-2 n t 
x_{j}\right)  \\
\eta_{t} &  =Ke^{K\phi}\left(  2\alpha\sum_{k=1}^{n}u_{k}u_{kt}+n
\left(  1-\left\vert x\right\vert ^{2}\right)  \right),
\end{align*}
so at $\left(x_{0},t_{0}\right)$ we have
\begin{align*}
0   \leq\frac{\eta_{t}-g^{ij}\eta_{ij}}{Ke^{K\phi}} 
&=\left(  2\alpha\sum
_{k=1}^{n}u_{k}u_{kt}+n \left(  1-\left\vert x\right\vert
^{2}\right)  \right)  \\
& \hspace{.75cm} -g^{ij}\left(  2\alpha\sum_{k=1}^{n}u_{kj}u_{ki}+2\alpha\sum_{k=1}%
^{n}u_{kij}u_{k}-2 nt\delta_{ij} \right)  \\
& \hspace{.75cm} -g^{ij}K\left(  2\alpha\sum_{k=1}^{n}u_{k}u_{ki}-2 n t x_{i}\right)
\left(  2\alpha\sum_{k=1}^{n}u_{k}u_{kj}-2 n t x_{j}\right)  .
\end{align*}
Using \eqref{Lu_k},
discarding the final term, and diagonalizing the Hessian (which will change the direction of the gradient at $(x_{0},t_{0})$ but not its magnitude) we obtain
\begin{align*}
0 &  \leq n \left(  1-\left\vert x\right\vert ^{2}\right)
-2\alpha\sum\frac{\lambda_{i}^{2}}{1+\lambda_{i}^{2}}+2 n t\sum\frac
{1}{1+\lambda_{i}^{2}},
\end{align*}
which can be written as
\[
2\alpha\sum\frac{\lambda_{i}^{2}}{1+\lambda_{i}^{2}}\leq\ 3n.
\]
Therefore, we conclude at $\left(  x_{0},t_{0}\right)$ that any eigenvalue $\lambda_i$ satisfies the bound
\[
\frac{\lambda_{i}^{2}}{1+\lambda_{i}^{2}}\leq\frac{3n}{2\alpha},%
\]
which is equivalent to 
\[
\lambda_{i}^{2}\leq\frac{3n}{2\alpha- 3n}.
\]
Thus,
\begin{align*}
V\left(  x_{0},t_{0}\right)   &  =\prod\left(  1+\lambda_{i}^{2}\right)
^{1/2} \\ &\leq\left[  1+\left( \frac{3n}{2\alpha- 3n}\right)  \right]  ^{n/2}\\
&  =\left(  \frac{2\alpha}{2\alpha-3n }\right)
^{n/2},
\end{align*}
and so
\[
b\left(  x_{0},t_{0}\right)  \leq \left(  \frac{2\alpha}{2\alpha-3n }\right)
^{1/2}.
\]

Now, for all $\left(  x,t\right)$ we have $\phi \leq 1$ 
and hence%
\[
h\left(  x_{0},t_{0}\right)  \leq e^{K} 
\left(  \frac{2\alpha}{2\alpha-3n }\right)
^{1/2},
\]
from which it follows that
\[
h(0,T)\leq e^{K}
\left(  \frac{2\alpha}{2\alpha-3n }\right)
^{1/2}.
\]
Finally, observe that
\[
\phi(0,\frac{1}{n})\geq-\alpha\gamma+1,
\]
which yields
\[
b(0,\frac{1}{n})\leq\frac{e^{K}
\left(  \frac{2\alpha}{2\alpha-3n }\right)
^{1/2}}{e^{K\left(  1-\alpha\gamma\right)  }-1}.
\]
Therefore, the maximum eigenvalue $\lambda_{\max}$ satisfies
\[
1+\lambda_{\max}^{2}\leq\frac{e^{2nK}\left(  \frac{2\alpha}{2\alpha-3n }\right)  ^{n}}{\left(  e^{K\left(  1-\alpha
\gamma\right)  }-1\right)  ^{2n}}.%
\]
\end{proof}

We now combine Proposition \ref{Hb} with Corollary \ref{coro} to obtain a Hessian bound from the oscillation. \

\begin{proposition}\label{main}
Suppose that $u$ is a convex solution of (\ref{ut2})
on $B_{2R+1}(0)\times\lbrack0,\frac{1}{n}]$ with%
\[
R=3\sqrt{n}
\]
and $u$ satisfies
\[
-1\leq u\leq1
\]
on $B_{2R+1}\times\left\{  0\right\}  $. \ \ Then
\[
D^{2}u(0,\frac{1}{n})\leq C(n).
\]

\end{proposition}

\begin{proof}
By Corollary \ref{coro}, we have the bound
\[
\max_{B_{1}(0)\times\lbrack0,\frac{1}{n}]}|Du| \leq\frac{1}{3\sqrt{n}}\left[
2+\arctan\left(  \frac{\pi}{9n}\right)  \right].
\]
We now apply the Hessian bound of Proposition \ref{Hb} with
\begin{align*}
T &  =\frac{1}{n}\\
\gamma &  =\frac{1}{9n}\left(  2+\arctan\left(  \frac{\pi}{9n}\right)
\right)  ^{2}<\frac{0.61}{n}\\
\alpha &  =1.6n\\
K &  =1
\end{align*}
to conclude
\[
\lambda_{\max}^{2}(0,\frac{1}{n})\leq\frac{e\left(  16\right)  ^{n}}{\left(
e^{\left(  1-\alpha\gamma\right)  }-1\right)  ^{2n}}=C(n).
\]

\end{proof}

\section{Proof of the Liouville Theorem}

To prove our main theorem we will show that for any $\left(  x_{0},t_{0}\right)$ we have
$D^{2}u(x_0,t_0)\leq C(n).$ \ Without loss of generality we may choose $t_{0}=0.$  For
fixed $x_{0}$, consider
\[
\tilde{u}_{\lambda}(x,t)=\frac{1}{\lambda^{2}}u\left(  \lambda(x-x_{0}%
),\lambda^{2}t\right),
\]
which restricts to a solution of (\ref{ut}) on $B_{6\sqrt{n}+1}%
(0)\times\lbrack-\frac{1}{n},0]$ with
\begin{equation}\label{tilde{u} vs u}
\sup_{x\in B_{6\sqrt{n}+1}(0)}\left\vert \tilde{u}_{\lambda}(x,-\frac{1}%
{n})\right\vert \leq\frac{1}{\lambda^{2}}\sup_{z\in B_{\left(  6\sqrt
{n}+1\right)  \left(  \lambda+\left\vert x_{0}\right\vert \right)  }%
(0)}\left\vert u\left(  z,-\lambda^{2}\frac{1}{n}\right)  \right\vert .
\end{equation}
By the growth at antiquity condition, for $\lambda$ large enough
and $\varepsilon_{0}<1$ small enough  we have
\begin{align*}
\sup_{z\in B_{\left(  6\sqrt{n}+1\right)  \left(  \lambda+\left\vert
x_{0}\right\vert \right)  }(0)}\left\vert u\left(  z,-\lambda^{2}\frac{1}%
{n}\right)  \right\vert  & \leq\frac{1}{\left(  6\sqrt{n}+2-\varepsilon
_{0}\right)  ^{2}}\left(  \left\vert \left(  6\sqrt{n}+1\right)  ^{2}\left(
\lambda+\left\vert x_{0}\right\vert \right)  \right\vert ^{2}+R_{0}\right)
\\
& \leq\frac{1}{1+\varepsilon_{1}}\lambda^{2}\left[  \left(  1+\frac{\left\vert
x_{0}\right\vert }{\lambda}\right)  ^{2}+\frac{R_{0}}{\lambda^{2}}\right] ,
\end{align*}
so, by choosing $\lambda>>\left\vert x_{0}\right\vert ,R_{0}$, we obtain from \eqref{tilde{u} vs u} %
\[
\sup_{x\in B_{6\sqrt{n}+1}(0)}\left\vert \tilde{u}_{\lambda}(x,-\frac{1}%
{n})\right\vert \leq1.
\]
We finally apply Proposition \ref{main} to conclude%
\[
D^{2}u(x_{0},0)=D^{2}u_{\lambda}(0,0)\leq C(n).
\]
Since this bound holds at any arbitrary point, it holds
everywhere, and we may apply \cite[Proposition 2.1]{YYmcf}, which states that any
convex ancient entire solution with a uniform Hessian bound must be a quadratic polynomial. This completes the proof.

\bibliography{Library}
\bibliographystyle{amsalpha}

\end{document}